\documentclass[a4paper,12pt]{amsart}
\usepackage{amssymb}
\usepackage{amsmath}
\usepackage{ifthen}
\usepackage{graphicx}
\usepackage{subfigure}
\usepackage{float}
\usepackage{geometry}
\usepackage{amsfonts}
\usepackage{amscd}
\usepackage{amsxtra}
\usepackage{color}

\newtheorem{theorem}{Theorem}
\newtheorem{corollary}{Corollary}
\newtheorem{lemma}{Lemma}

\newtheorem{prob}{Problem}

\newtheorem{conj}{Conjecture}
\theoremstyle{definition}
\newtheorem{example}{Example}

\newcounter {own}
\def\theown {\thesection       .\arabic{own}}

\newenvironment{rem}{%
\bigskip
\noindent \textsl{{\sl Remark. }}}{\bigskip}
\newenvironment{rems}{%
\bigskip
\noindent \textsl{{\sl Remarks. }}}{\bigskip}

\newenvironment{pf}[1][]{%
 \vskip 3mm
 \noindent
 \ifthenelse{\equal{#1}{}}%
  {{\slshape Proof. }}%
  {{\slshape #1.} }%
 }%
{\qed\bigskip}

\newcounter{alphabet}
\newcounter{tmp}
\newenvironment{Thm}[1][]{\refstepcounter{alphabet}%
\bigskip%
\noindent%
{\bf Theorem \Alph{alphabet}}%
\ifthenelse{\equal{#1}{}}{}{ (#1)}%
{\bf .} \itshape}{\vskip 8pt}

\makeatletter
\newcommand{\Ref}[1]{\@ifundefined{r@#1}{}{\setcounter{tmp}{\ref{#1}}\Alph{tmp}}}
\makeatother

\newcommand{\IC}{{\mathbb C}}
\newcommand{\ID}{{\mathbb D}}
\newcommand{\IT}{{\mathbb T}}

\newcommand{\real}{{\operatorname{Re}\,}}

\def\be{\begin{equation}}
\def\ee{\end{equation}}

\newcommand{\bee}{\begin{enumerate}}
\newcommand{\eee}{\end{enumerate}}

\newcommand{\blem}{\begin{lemma}}
\newcommand{\elem}{\end{lemma}}
\newcommand{\bthm}{\begin{theorem}
}
\newcommand{\ethm}{\end{theorem}
}
\newcommand{\bcor}{\begin{corollary}
}
\newcommand{\ecor}{\end{corollary}
}
\newcommand{\beg}{\begin{example}}
\newcommand{\eeg}{\end{example}}
\newcommand{\begs}{\begin{examples}}
\newcommand{\eegs}{\end{examples}}
\newcommand{\bdefe}{\begin{defin}}
\newcommand{\edefe}{\end{defin}}
\newcommand{\bprob}{\begin{prob}}
\newcommand{\eprob}{\end{prob}}
\newcommand{\bei}{\begin{itemize}}
\newcommand{\eei}{\end{itemize}}

\newcommand{\bcon}{\begin{conj}}
\newcommand{\econ}{\end{conj}}
\newcommand{\bcons}{\begin{conjs}}
\newcommand{\econs}{\end{conjs}}
\newcommand{\bprop}{\begin{propo}}
\newcommand{\eprop}{\end{propo}}
\newcommand{\br}{\begin{rem}}
\newcommand{\er}{\end{rem}}
\newcommand{\brs}{\begin{rems}}
\newcommand{\ers}{\end{rems}}
\newcommand{\bo}{\begin{obser}}
\newcommand{\eo}{\end{obser}}
\newcommand{\bos}{\begin{obsers}}
\newcommand{\eos}{\end{obsers}}
\newcommand{\bpf}{\begin{pf}}
\newcommand{\epf}{\end{pf}}
\newcommand{\ba}{\begin{array}}
\newcommand{\ea}{\end{array}}
\newcommand{\beq}{\begin{eqnarray}}
\newcommand{\beqq}{\begin{eqnarray*}}
\newcommand{\eeq}{\end{eqnarray}}
\newcommand{\eeqq}{\end{eqnarray*}}

\newcommand{\ra}{\rightarrow}

\newcommand{\ds}{\displaystyle}

\newcounter{minutes}
\divide\time by 60
\newcounter{hours}
\multiply\time by 60 \addtocounter{minutes}{-\time}

\begin{document}
\bibliographystyle{amsplain}
\title[]{Dirichlet problem, Univalency and Schwarz Lemma for Biharmonic Mappings}

\thanks{
File:~\jobname .tex,
          printed: \number\day-\number\month-\number\year,
          \thehours.\ifnum\theminutes<10{0}\fi\theminutes}
\author[Z. Abdulhadi]{Zayid  Abdulhadi}
\address{Z. Abdulhadi, Department of Mathematics,
American University of Sharjah, UAE-26666.}
\email{zahadi@aus.edu}

\author[Y. Abu Muhanna]{Yusuf  Abu Muhanna}
\address{Y. Abu Muhanna, Department of Mathematics,
American University of Sharjah, UAE-26666.}
\email{ymuhanna@aus.edu}

\author[S. Ponnusamy]{Saminathan Ponnusamy $^\dagger $
}
\address{S. Ponnusamy, Stat-Math Unit,
Indian Statistical Institute (ISI), Chennai Centre,
110, Nelson Manickam Road,
Aminjikarai, Chennai, 600 029, India.}
\email{samy@isichennai.res.in, samy@iitm.ac.in}

\keywords{Harmonic and biharmonic mappings, Dirichlet problem, univalent, convex, Green function, Laplacian, Schwarz Lemma. \\
$
^\dagger$ {\tt  The third author is currently on leave from IIT Madras}
}

\begin{abstract}
In this paper, we shall discuss the family of biharmonic mappings for which maximum principle holds. As a
consequence of our study,  we present Schwarz Lemma for the family of biharmonic
mappings. Also we discuss the univalency of certain class of biharmonic mappings.
\end{abstract}

\maketitle \pagestyle{myheadings}
\markboth{Z. Abdulhadi, Y. Abu Muhanna and S. Ponnusamy}{Dirichlet problem, Univalency and Schwarz Lemma for Biharmonic Mappings}

\section{Introduction}\label{AAP7-sec1}

Investigation of biharmonic mappings in the context of geometric function theory is started only recently. Indeed
several important properties of biharmonic mappings are obtained in
\cite{AbAb-08,AbAbKh-2004,AbAbKh-2005,AbAbKh-2006,AbAli-2013,ChPoWa-2009,ChHerm-2007}
and these mappings were also generalized by some others, see for example \cite{AmGaPo-2017,Li-2013} and the references therein.
In the point of view of applied mathematics, biharmonic mappings arise naturally in fluid dynamics and elasticity problems,
and have  important applications in engineering and biology (see \cite{HaBre-1965,Lab-1964}). From the point of view of differential geometry,
biharmonic mappings are closely related to the theory of Laguerre minimal surfaces. For details, we refer to
\cite{AbAli-2013,BobPink-1996,Blas-1924,Blas-1925,PetPot-1998,PotGroMit-2009}). Thus, various kinds of problems
for harmonic and biharmonic (and more generally, polyharmonic and polyanalytic) mappings have  widely been investigated. In this article,
we are mainly concerned with maximum principle and Schwarz lemma for biharmonic mappings. We need some preparations before we address
our main issues concerning biharmonic mappings.

A real-valued $C^2$-function $u$ is harmonic in an open set in $\IC$ if
$\bigtriangleup u=0$ there, and is subharmonic if $\bigtriangleup u\geq 0$, where
$$\bigtriangleup =4\frac{\partial ^{2}}{\partial z\partial {\overline{z}}}
=4\bigtriangleup_z:=\frac{\partial ^{2}}{\partial x^{2}}+\frac{\partial ^{2}}{\partial y^{2}}, \quad z=x+iy,
$$
denotes the Laplace operator. Then the classical maximum principle for subharmonic functions states
that if $D$ is a bounded domain in $\IC$, and $u$ is continuous on the closure of $D$, then
$$\left \{0\leq \bigtriangleup u(z) ~\mbox{on $D$}~\mbox{and}~ u(z)\leq 0 ~\mbox{on $\partial D$}\right \} \Longrightarrow
 u(z)\leq 0 ~\mbox{on $D$},
$$
with equality at some point in $D$ if and only if $u(z)=0$  on $D$.

In 1908, Jacques Hadamard suggested the possibility of a maximum principle for the bilaplacian $\bigtriangleup ^2$.
Recall that a four times continuously differentiable real-valued function $u$ on a domain $\Omega$ is biharmonic
if $\bigtriangleup ^2u=\bigtriangleup (\bigtriangleup u)=0$, and
sub-biharmonic if $\bigtriangleup ^2u\leq 0$ (one should think of $\bigtriangleup$ as a negative operator, which is the reason why the
inequality is reversed as compared with the definition of subharmonic functions).
Obviously, every harmonic function is biharmonic but not necessarily the converse.

A complex-valued $C^2$-function $f$ in a simply connected domain in $\mathbb{C}$ is harmonic
if $\bigtriangleup f=0$ there. It is almost obvious that the mapping $f$ has a canonical decomposition $f=h+\overline{g},$
where $h$ and $g$ are analytic (holomorphic) there. Similarly,
a four times continuously differentiable complex-valued function $F$ in a simply connected domain is biharmonic if $\bigtriangleup ^2F=0$ there.
It is easy to see that every biharmonic mapping $F$ has the representation
\be\label{AAP7-eq0}
F(z)=|z|^{2}A(z)+B(z),
\ee
where $A$ and $B$ are harmonic there.

Let $D$ be a circular disk in the plane domain $\Omega$ and $u$ is a $C^1$-smooth function on the closure of $D$.
Then a variant of maximum principle for the bilaplacian takes the following formulation:
if $D$ is a bounded domain in $\IC$, and $u$ is continuous on the closure of $D$, then
$$\left \{ \left .\bigtriangleup ^2 u\right |_{D} \leq 0,~ \left . u \right |_{\partial D}\leq 0,
~\mbox{and}~  \left . \frac{\partial u}{\partial n} \right |_{\partial D}\leq 0
\right \} \Longrightarrow \left . u \right |_{D}\leq 0,
$$
whereby the normal derivative is calculated in the interior direction. Moreover, unless $u(z)=0$ on $D$, we have
$u(z)<0$ on $D$.

The paper is organized as follows. In Section \ref{AAP7-sec2}, we consider maximum principle for biharmonic mappings. This has
led to investigate a biharmonic analog of Rad\'{o}-Kneser-Choquet Theorem (see Theorem \Ref{AAP7-TheoB} and Problem \ref{AAP-prob1})
and in support of this proposal, we present a set of examples. In Section \ref{AAP7-sec3}, we prove a version of biharmonic Schwarz lemma.

In order to  motivate our investigation, it is more appropriate to express the above maximum principle in terms of the biharmonic Green
function on $D$ and this requires some preparation. This will be done in Section \ref{AAP7-sec2}.
Actually, Hadamard suggested that the maximum principle for $\bigtriangleup ^2$
should be valid for general domains than just disks, including all convex domains with smooth boundary. However, it was shown that this
was not the case. In 1951, Garabedian \cite{Gar-1951} has shown that this fails when $D$ is an ellipse, provided that the ratio
of the major axis to minor axis exceeds a certain critical value $\epsilon \approx 1.5933$ and later this value has been improved
to $\epsilon \approx 1.1713$. This observation shows that within the family
of ellipses, we cannot deviate too far from circles and keep the maximum principle for $\bigtriangleup ^2$ valid.
Thus, circular disks are somehow natural for the bilaplacian which is also corroborated in terms of the
work of Loewner \cite{Loe-1953}.


\section{Drichlet problem  for biharmonic mappings}\label{AAP7-sec2}

\subsection{Preliminaries and basic tools}
We consider biharmonic mappings defined on the unit disk $\ID:=\{z\in\IC:\, |z|<1\}$. We denote by $\IT$, the unit circle $\{z\in\IC:\, |z|=1\}$
and $\overline{\ID}:=\ID\cup\IT$, the closed unit disk. Throughout this paper, it is more convenient to consider biharmonic mappings of the form
\be\label{AAP7-eq1}
F(z)=H(z)+(1-|z|^2)h(z),
\ee
where $H$ and $h$ are harmonic in $\ID$. In this article, we study the space of solutions
of the Dirichlet problem:
\be\label{AAP7-eq1a}
\left \{\begin{array}{l} \mbox{Solve the biharmonic equation}~ \bigtriangleup ^{2}u=0 ~\mbox{ on $\ID$}\\
\mbox{subject to the conditions}\\
\ds u=\varphi \mbox{ on $\IT$} ~\mbox{ and }~
\frac{\partial u}{\partial n}=\psi \mbox{ on $\IT$},
\end{array}\right.
\ee
where the normal derivative is taken in the exterior direction.
It is known that every solution of this problem is of the form \eqref{AAP7-eq1} and thus, it is biharmonic in $\ID$.
Moreover, the biharmonic Green function  for the operator $\bigtriangleup ^2 $ in the unit disk $\ID$ (with Dirichlet boundary conditions)
is the function $\Gamma (z,\zeta )$ defined by the expression
$$\Gamma (z,\zeta )=|z-\zeta |^{2}G(z,\zeta )
+\left( 1-|z|^{2}\right) \left( 1-|\zeta |^{2}\right),\quad  (z,\zeta)\in \ID\times \ID,
$$
where
$$G(z,\zeta )=\log \left\vert \frac{z-\zeta }{1-z\overline{\zeta }}\right\vert ^{2},\quad  (z,\zeta)\in \ID\times \ID,
$$
stands for the usual Green function for the Laplacian $\bigtriangleup $ in the unit disk $\ID$.
A calculation shows that $\Gamma (z,\zeta )>0$ on the bidisk $\ID\times \ID$ (see for instance
\cite[Proposition 2.3]{AlRi-1996}).  For a fixed $\zeta\in\ID$, the biharmonic Green function $\Gamma (.\,,\zeta )$ solves the
following boundary value problem:
$$\left \{\begin{array}{l}
 \bigtriangleup ^{2}_z \Gamma (z,\zeta )=\delta_\zeta (z) ~\mbox{ for $z\in \ID$,}\\
\ds \Gamma (z,\zeta )= 0 \mbox{ on $\IT$,} \\
\ds \partial_{n(z)} \Gamma (z,\zeta )=0 \mbox{ on $\IT$,}
\end{array}\right.
$$
where $\partial_{n(z)}=\frac{\partial   }{\partial n(z)}$ denotes the inward normal derivative being taken with respect
to the boundary variable $z\in\IT$  in the interior direction in the sense of distributions.
Then any $u(z)\in C^{4},$ the solution of the above Dirichlet problem can be  captured from the boundary: For $z\in\ID$,
\begin{eqnarray}\label{AAP7-eq2}
u(z) &=&\int_{\ID} \Gamma (\zeta,z )\bigtriangleup ^{2}u(\zeta )\,dA(\zeta )   \\
\nonumber
&& +\frac{1}{2} \int_{\IT}\left[ \bigtriangleup _{\zeta }\Gamma (\zeta ,z )
\frac{\partial u}{\partial n(\zeta )}(\zeta ) \,d\sigma (\zeta )
-\frac{\partial}{\partial n(\zeta )}\bigtriangleup _{\zeta }\Gamma (\zeta,z)  u(\zeta ) \right]d\sigma (\zeta ) ,
\end{eqnarray}
where $dA(\zeta ) = (1/\pi)dx dy$ denotes the normalized Lebesgue area measure on the unit disk $\ID$,
and, for $z = e^{i\theta}$, we write $d\sigma (\zeta ) = (1/2\pi)d\theta$ for the normalized arc length measure on the unit circle $\IT$.

A computation gives
$$ \bigtriangleup _{\zeta}\Gamma (\zeta ,z )= G(\zeta ,z ) + H(\zeta ,z)~\mbox{ for }~ (\zeta,z)\in \ID \times \ID ,
$$
where $H(\zeta ,z )$ is the harmonic compensator defined by
$$H(\zeta ,z )=(1-|z|^2)P(z,\zeta) =\frac{(1-|z|^2)^2}{|1-\overline{z}\zeta|^2},
\quad  (\zeta,z)\in \IT\times \ID
$$
in which $P(z,\zeta)$ stands for the Poisson kernel for the unit disk. The function $H(\zeta ,z )$
is harmonic in its first argument and is biharmonic in its second argument. Clearly it is not symmetric in its arguments.
Another computation gives that the function
$$F(\zeta ,z )=-\frac{1}{2}\partial_{n(\zeta)} \bigtriangleup _{\zeta }\Gamma (\zeta,z),  \quad  (\zeta,z)\in \IT\times \ID ,
$$
has the form
$$F(\zeta ,z )=\frac{1}{2} \frac{(1-|z|^2)^2}{|1-\overline{z}\zeta|^2} +\frac{1}{2}\frac{(1-|z|^2)^3}{|1-\overline{z}\zeta|^4}  ,
\quad  (\zeta,z)\in \IT\times \ID.
$$
It turns that $F(\zeta ,z )$ is biharmonic in its second argument and is a certain biharmonic Poisson kernel studied by
Abkar and  Hedenmalm \cite{AH-2001}. We refer to \cite{DuSch-2004,HedelKoZhu-2000} for a general reference on this topic
where one can also obtain primary connection between the Green function for $\bigtriangleup ^{2}$ with Dirichlet boundary conditions
and the Bergmann spaces.

The formula \eqref{AAP7-eq2} implies the following well-known form of the maximum principle for ready reference.

\begin{Thm}\label{AAP7-TheoA}
If $u(z)\in C^{4}$ on $\overline{\ID}$ is real and subject to the
conditions:
$$\left \{\left. \bigtriangleup ^{2}u\right |_{D} \leq 0, ~\left. u\right |_{\IT}\leq 0, ~
\left. \frac{\partial u}{\partial n} \right |_{\IT}\leq 0 \right \},
$$
then $u\leq 0$ on $\ID.$
\end{Thm}

The proof of this theorem is apparent from \eqref{AAP7-eq1a}, \eqref{AAP7-eq2} and the fact that
$$\left \{\begin{array}{l}
\ds \Gamma (z,\zeta )>0~\mbox{ for $(z,\zeta )\in \ID\times \ID$},~ \\
\ds \bigtriangleup _{z}\Gamma (z,\zeta )>0
~\mbox{ for $(z,\zeta )\in \IT\times \ID$},\\ ~
\ds \frac{\partial  }{\partial n(z)} \bigtriangleup _{z}\Gamma (z,\zeta )<0~\mbox{ for $(z,\zeta )\in \IT\times \ID$}.
\end{array}
\right .
$$

\subsection{Dirichlet problem and the Univalency of Biharmonic mappings}

\begin{lemma}\label{AAP7-lem1}
If the solution $u(z)$ of the Dirichlet problem \eqref{AAP7-eq1a} belongs to $C^{4}(\overline{\ID})$ and is of the
form \eqref{AAP7-eq1}, i.e., $u(z)=H(z)+(1-|z|^{2})h(z)$,
then $H=\varphi $ on $\IT$ and $h=\frac{1}{2}(\psi +\frac{\partial H}{\partial n})$ on $\IT.$
\end{lemma}
\begin{proof}
It is easy to see that
$$\frac{\partial u}{\partial n}(z)=-2h(z)+\frac{\partial H}{\partial n}(z) \mbox{ on $\IT$}.
$$
By \eqref{AAP7-eq1a}, we have $H=\varphi $ on $\IT$ and $h=\frac{1}{2}(\psi +\frac{\partial H}{\partial n})$ on $\IT.$
\end{proof}

From now onwards, we choose $\psi =0$. Thus, we get the following family which contains solutions of \eqref{AAP7-eq1a}:
$${\mathcal F} =\left\{u:\, u(z)=H(z)+(1-|z|^{2})\frac{1}{2}\frac{r\partial H}{\partial n}(z),  ~\mbox{$H$ is harmonic in $\ID$}
\right\}.
$$

\begin{lemma}\label{AAP7-lem2}
If $u\in {\mathcal F} \cap C^{1}(\overline{\ID})$, then $u$ is a
solution of  \eqref{AAP7-eq1a} with $\varphi (z)=H(z)$ on $\IT$, $\psi =0$ and $h=\frac{1}{2}\frac{%
r\partial H}{\partial n}$ on $\IT.$
\end{lemma}
\begin{proof}
The Dirichlet problem \eqref{AAP7-eq1a} has a unique solution with the given boundary conditions.
Then Lemma \ref{AAP7-lem1} implies the desired result.
\end{proof}

Following is a maximum principle for functions in the family ${\mathcal F}.$

\begin{lemma}\label{AAP7-lem3}
Let $u\in {\mathcal F}\cap C^{1}(\overline{\ID})$ and real. If $u\leq M$ on $\IT$, then $u\leq M$ on $\ID.$
\end{lemma}
\begin{proof}
Lemma \ref{AAP7-lem2} implies that $\psi =0$, $\frac{\partial u}{\partial n}(z)=0$ and
thus, Theorem \Ref{AAP7-TheoA} implies the required conclusion.
\end{proof}

\begin{corollary}\label{AAP7-cor1}
{\rm (Maximum principle)} If $u\in {\mathcal F}\cap C^{1}(\overline{\ID})$ and $u:\,\IT\rightarrow \IT$ is complex-valued, then
$u$ maps $\ID$ into $\ID.$
\end{corollary}
\begin{proof}
Let $v(z)=\real (\lambda u(z))$, where $|\lambda |=1.$ Then $v\in {\mathcal F} $,
and $v(z)<1$ for all $z\in\ID$ for each $\lambda\in \IT$. Fix $z$ and
choose $\lambda$ so that $v(z)= |u(z)|$ and, thus, $v(z)\leq |u(z)|\leq 1$ on $\IT.$ Finally,
Lemma \ref{AAP7-lem3} gives the desired result.
\end{proof}

We say that a harmonic mapping $f=h+\overline{g}$ is locally univalent and sense preserving in a simply connected domain $\Omega$
if and only if its Jacobian $J_{f} (z)$ is positive there, where
$ J_{f} (z)=|f_{z}(z)|^{2}-|f_{\overline{z}}(z)|^{2}=|h'(z)|^{2}-|g'(z)|^{2}.
$
According to Lewy's theorem \cite{Le-1936}, $f$ is locally univalent and sense-preserving in $\Omega$ if and only if $
|g'(z)|<|h'(z)|$ in $\Omega$. See \cite{Clunie-Small-84, Du-2004,SaRa2013} for a detailed information on harmonic mappings and its important
geometric subfamilies.


\begin{Thm}\label{AAP7-TheoB}
{\rm (Rad\'{o}-Kneser-Choquet Theorem)}  Let $f^*$ be a homeomorphism
from $\IT$ onto $\partial \Omega,$ where $\Omega $ is a bounded convex domain.
Then its harmonic extension
$$ f(z) = \frac{1}{2\pi}\int_0^{2\pi} \frac{1-|z|^2}{|1-e^{-i\theta}z|^2}f^*(e^{i\theta})\, d\theta
$$
is univalent in $\ID$ and defines a harmonic mapping of $\ID$ onto $\Omega$.
\end{Thm}

Let us now suppose that $H(z)=w_1(z)+\overline{w_2(z)},$  where both $w_{1}$ and $w_{2}$ are analytic in $\ID$.
A natural question is to ask whether Rad\'{o}-Kneser-Choquet Theorem continues to hold for ${\mathcal F}$.

\bprob\label{AAP-prob1}
Suppose that $H:\,\IT\rightarrow \IT$ is bijective, and $H(z)=w_1(z)+\overline{w_2(z)},$  where both $w_{1}$ and $w_{2}$ are
analytic in $\ID$. If
$$u(z)=H(z)+(1-|z|^{2})\frac{1}{2}(zw_{1}'(z)+\overline{zw_{2}'(z)})
$$
is the biharmonic extension to $\ID$, is $u$ univalent on $\ID$?
\eprob

Let us continue the discussion with a couple of examples to motive this problem in a slightly general format.
Our first example deals with the case where $H(z)=z$.

\begin{example}
Consider $u(z)=z+0.5(1-|z|^{2})z.$ Then it is easy to see that it is univalent in $\ID.$ Indeed for $z_1,z_2\in\ID$, $u(z_1)=u(z_2)$ gives
$$z_1(3-|z_1|^2)=z_2(3-|z_2|^2), ~\mbox{ i.e. }~\frac{z_1}{z_2}=\frac{3-|z_2|^2}{3-|z_1|^2},
$$
since there is nothing to prove if $z_1=0$ or $z_2=0$. So we may assume that $z_1,z_2\in\ID\backslash\{0\}$.
Then the last relation obviously shows that $r=z_1/z_2$ is real and positive. Again, without loss of generality, we can assume that $r\in (0,1]$ so that $z_1=rz_2$ and thus, the last relation reduces to
$$r=\frac{3-|z_2|^2}{3-r^2|z_2|^2},~\mbox{ i.e. }~(1-r)[3-|z_2|^2(1+r+r^2)]=0,
$$
which clearly gives that $r=1$ and thus, $z_1=z_2$. This proves the univalency of $u(z)$ in $\ID$.
More generally, it is easy to see that for each $\alpha \in (0,1/2]$, the biharmonic mapping $u_\alpha (z)=z +\alpha (1-|z|^2)z$ is univalent in $\ID$.
%
\end{example}

In order to provide a proof of the next example, we need the following reformulated version of \cite[Theorem~1.1]{AmGaPo-2017}.

\begin{Thm} \label{AAP-Theo0}
Let $u(z)= |z|^2F_1(z)+F_2(z)$ be biharmonic in $\mathbb{D}$ and univalent in a neighborhood of the origin, where
$$F_{j}(z)=\sum_{n=1}^{\infty} \left (a_n^{(j)} z\thinspace^n+b_n^{j}{\overline{z}}\thinspace^n\right ) ~\mbox{ ($j=1,2$)}
$$
are harmonic in $\mathbb{D}.$ Then the function $u(z)$ is univalent in $\mathbb{D}$ if and only if for each
$z\in\mathbb{D}\backslash\{0\}$ and $t\in\left(0,\pi/2\right]$ the following condition holds:
\begin{equation} \label{eq0}
 \sum\limits_{n=1}^{\infty}\left(a_n^{(2)} z^n - b_n^{(2)}{\overline{z}}^n\right)
\frac{\sin nt}{\sin t} + |z|^2\, \sum\limits_{n=1}^{\infty}\left(a_n^{(1)} z^n - b_n^{(1)}{\overline{z}}^n\right)
\frac{\sin nt}{\sin t}
\neq0.
\end{equation}
\end{Thm}

\begin{figure}
\begin{center}
\includegraphics[height=5.5cm, width=5.5cm, scale=1]{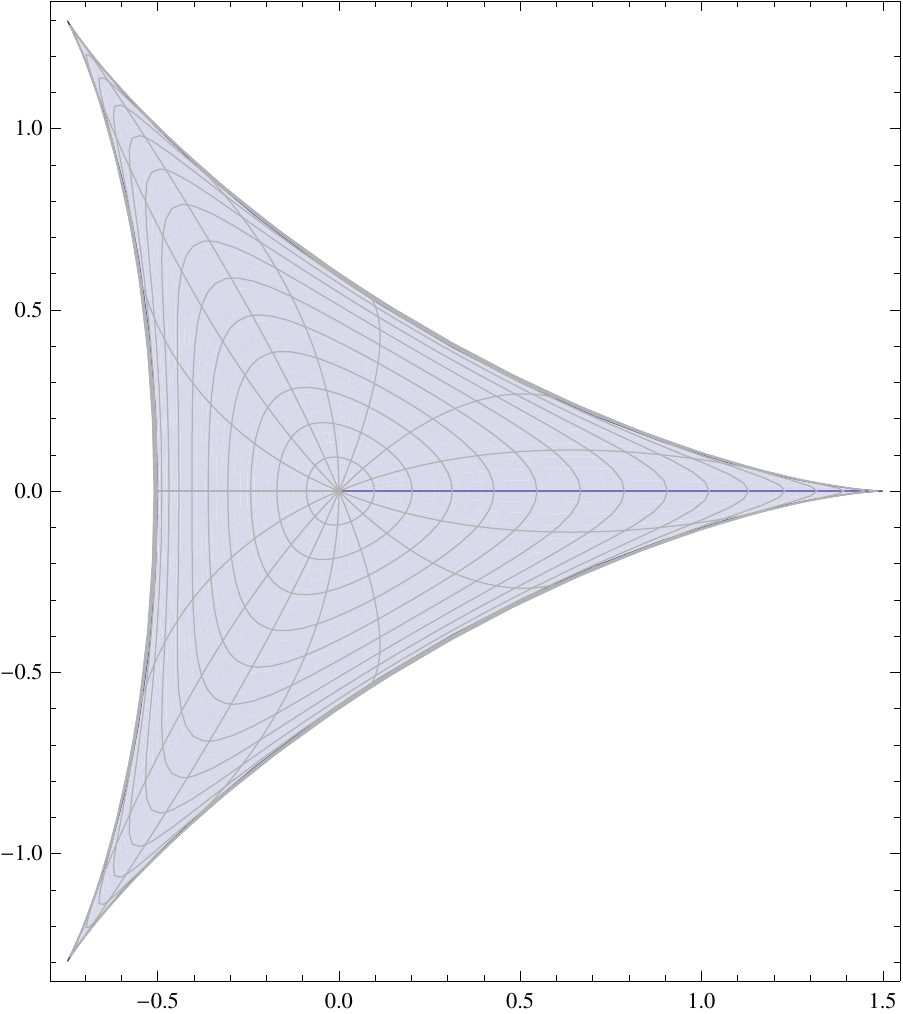}
\hspace{1cm}
\includegraphics[height=5.5cm, width=5.5cm, scale=1]{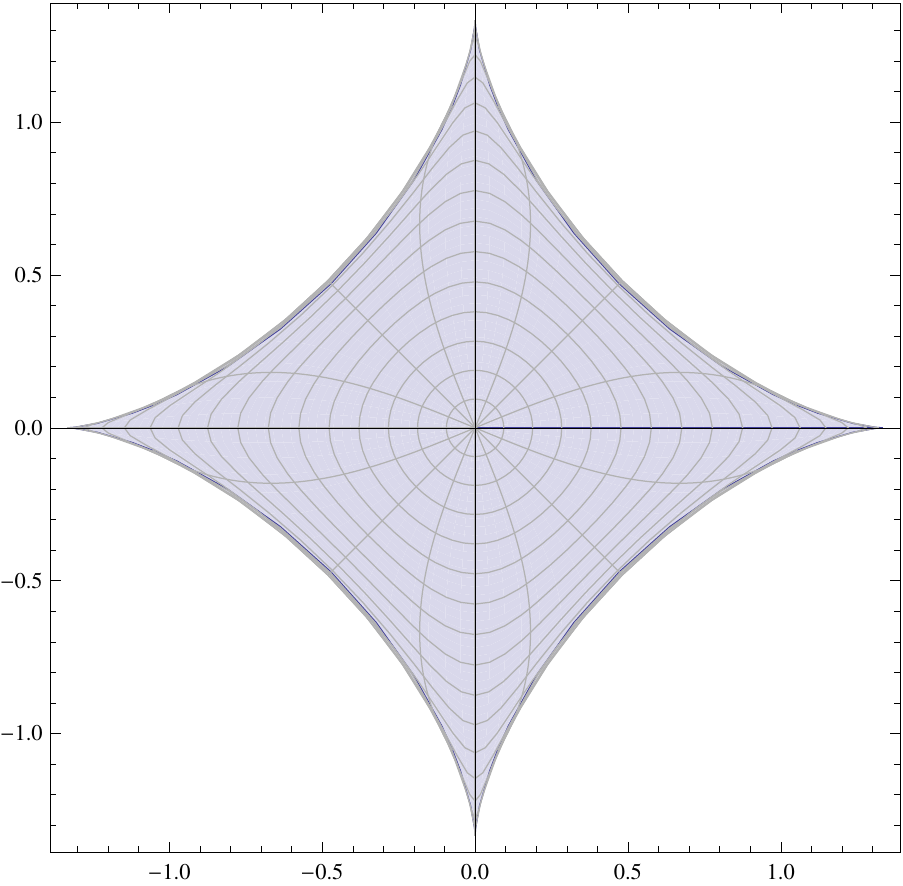}
\end{center}
$n=2$ \hspace{5cm} $n=3$

\vspace{0.25cm}

\begin{center}
\includegraphics[height=5.5cm, width=5.5cm, scale=1]{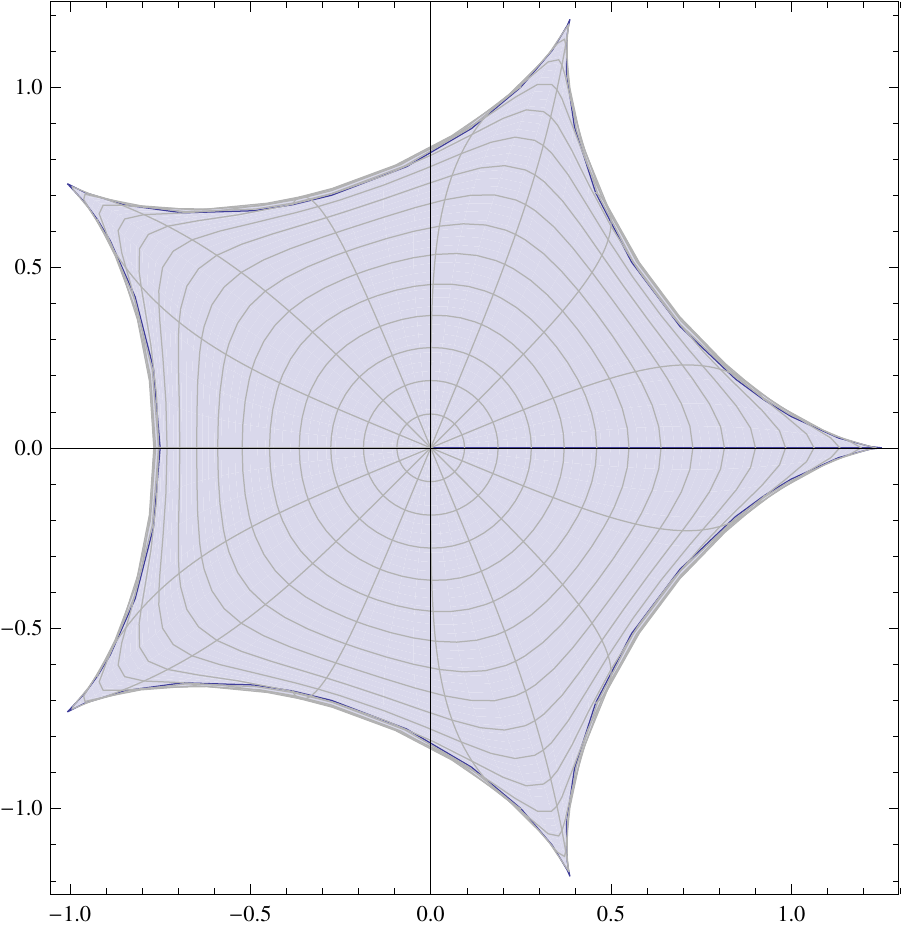}
\hspace{1cm}
\includegraphics[height=5.5cm, width=5.5cm, scale=1]{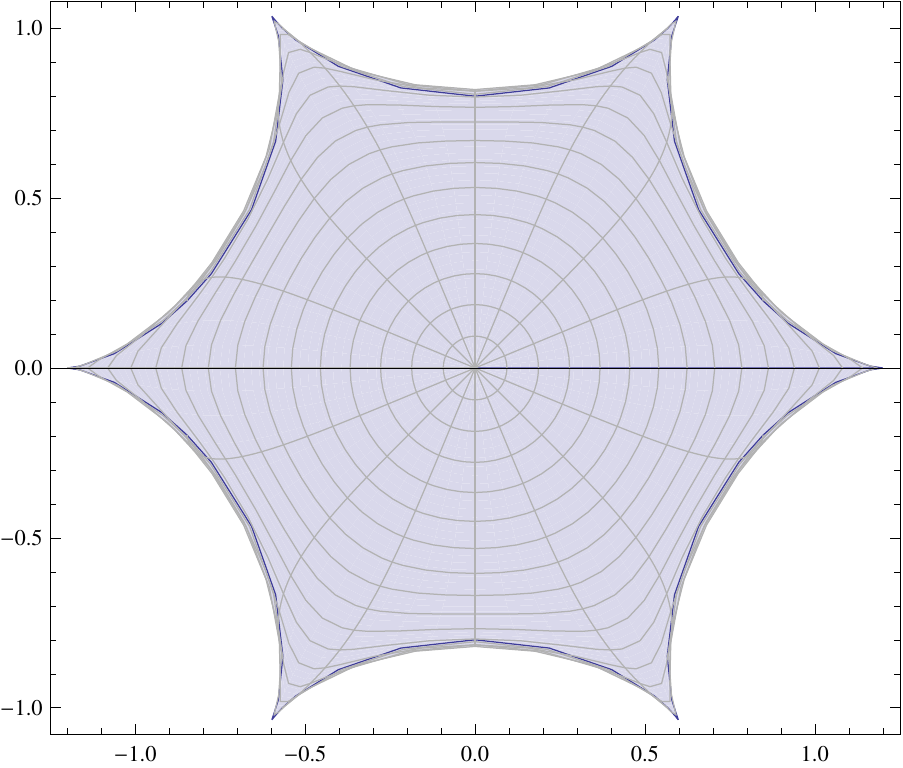}
\end{center}
$n=4$ \hspace{5cm} $n=5$

\vspace{0.25cm}

\begin{center}
\includegraphics[height=5.5cm, width=5.5cm, scale=1]{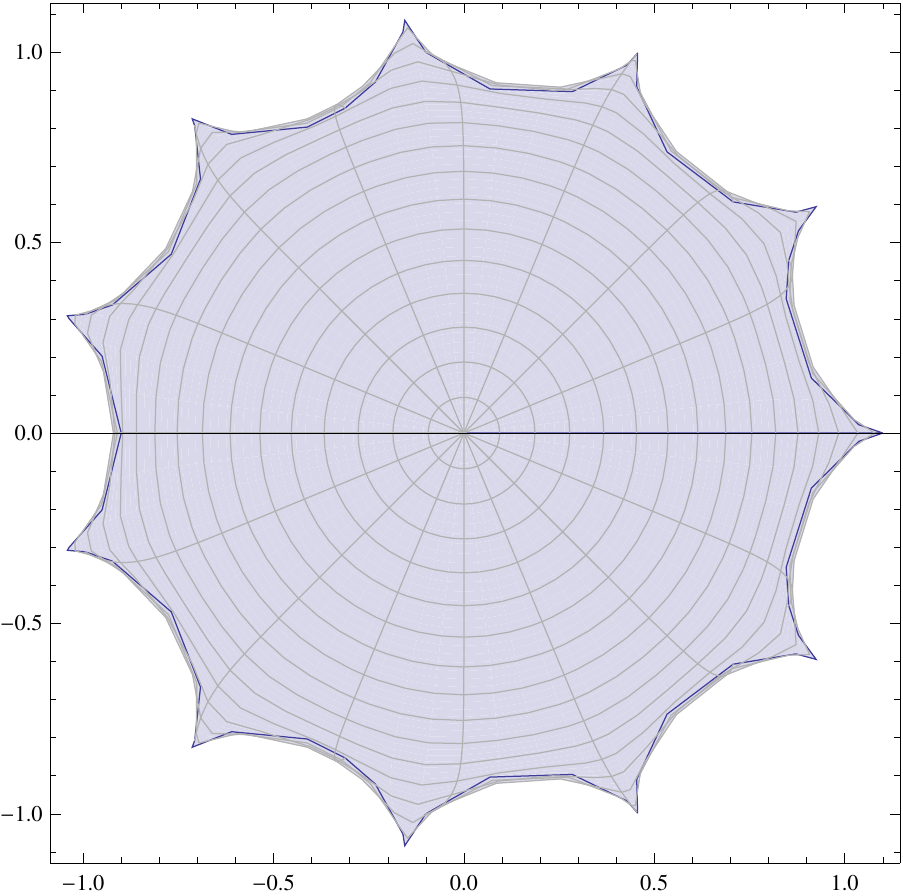}
\hspace{1cm}
\includegraphics[height=5.5cm, width=5.5cm, scale=1]{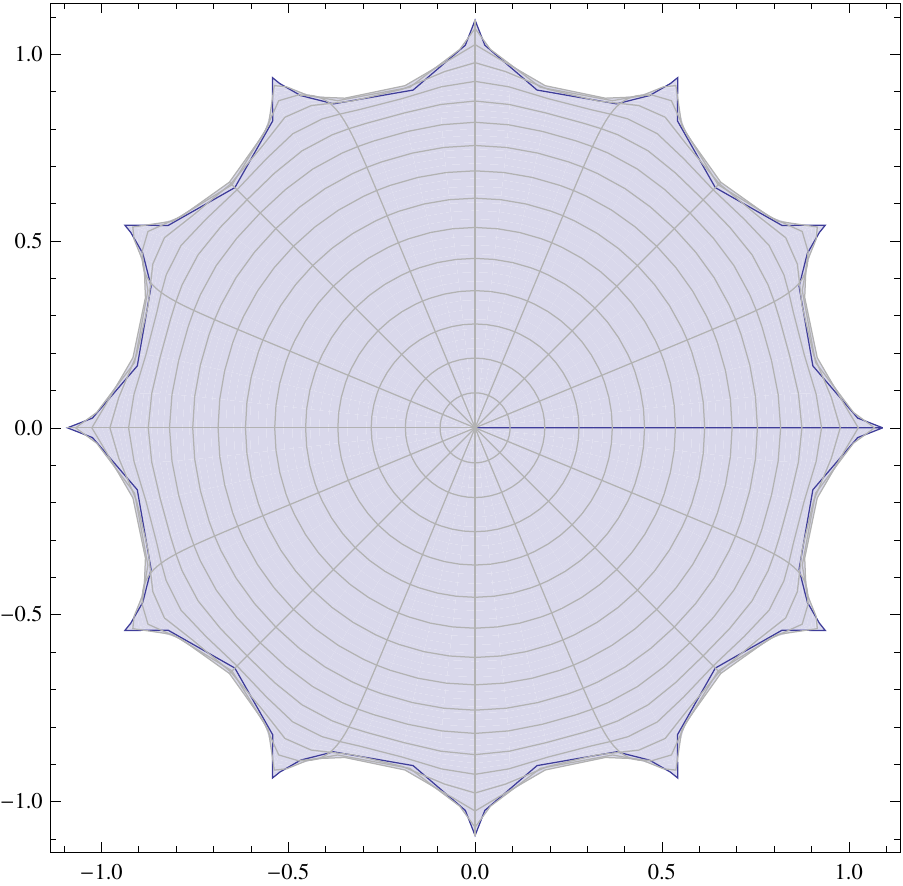}
\end{center}
$n=10$ \hspace{5cm} $n=11$

\caption{Graphs of $u(z)$ in $\protect \ID$ for certain values of $n$\label{Fig-u(n)}}
\end{figure}
\begin{example}
For $n\ge 2$, we consider the biharmonic mappings
$$ u(z)=z+\frac{\overline{z}^{n}}{n}+(1-|z|^{2})\frac{1}{2}(z+\overline{z}^{n})
$$
which may be rewritten as $u(z)= |z|^2F_1(z)+F_2(z)$, where
$$ F_1(z)=- \frac{1}{2}(z+\overline{z}^{n})~\mbox{ and }~ F_2(z)=\frac{3z}{2}+\left(\frac{n+2}{2n}\right )\overline{z}^{n}.
$$
It is easy to see that for each $n\ge 2$, the harmonic mapping $H(z)=z+(1/n)\overline{z}^{n}$ is univalent in  $\ID$.

We next show that  $u(z)$ is univalent in $\ID$. According to Theorem \Ref{AAP-Theo0}, it suffices to show that for each
$z\in\mathbb{D}\backslash\{0\}$ and $t\in\left(0,\pi/2\right]$, the following inequality holds:
$$\frac{3z}{2}-\left(\frac{n+2}{2n}\right )\overline{z}^{n}\frac{\sin nt}{\sin t} +
|z|^2\left( -\frac{z}{2}+\frac{\overline{z}^{n}}{2}\frac{\sin nt}{\sin t}
\right ) \neq 0,
$$
or equivalently
\be\label{AAP7-eq7}
A(z)=3z-|z|^2z \neq B(z)=\left( \frac{n+2}{n}-|z|^2\right )\overline{z}^{n}\frac{\sin nt}{\sin t}  .
\ee
The fact that $|\sin nt| \le n |\sin t|$  for all $t \in [0, \pi/2]$ and $n \geq 2$ (which may easily be verified by a method of induction), may
be used to verify the last relation. Because
$$|A(z)|\geq |z|(3-|z|^2) ~\mbox{ and }~ |B(z)|\leq \left( \frac{n+2}{n}-|z|^2\right )n|z|^{n},
$$
the relation \eqref{AAP7-eq7} holds for $z\in\mathbb{D}\backslash\{0\}$ and $t\in\left(0,\pi/2\right]$, provided that
\be\label{AAP7-eq8}
3-|z|^2>( n+2-n|z|^2)|z|^{n-1}.
\ee
For $n=2,3$, the inequality \eqref{AAP7-eq8} reduces to $(1-|z|)^2(3+2|z|)>0$ and $3(1-|z|^2)^2>0$, respectively. Thus, \eqref{AAP7-eq8} holds and so,
$u(z)$ is univalent in $\ID$ for $n=2,3$. The general case may be proved by the method of induction applied to \eqref{AAP7-eq8}. In fact, it is a simple
exercise to see that
$$(n+2-n|z|^2)|z|^{n-1}>(n+3-(n+1)|z|^2)|z|^{n}
$$
is equivalent to
$$(1-|z|)^2(n+2+(n+1)|z|)>0
$$
which obviously holds for $z\in\mathbb{D}\backslash\{0\}$ and for $z\in\mathbb{D}\backslash\{0\}$. Consequently, \eqref{AAP7-eq8} holds for all $n\geq 2$.
Hence, $u(z)$ is univalent in $\ID$ for all $n\geq 2$.
%
The graphs of $u(z)$, for certain values of $n\geq 1$, are shown in Figures \ref{Fig-u(n)}.
\end{example}

\begin{example}
Next we consider the harmonic function $H(z)=w_1(z)+\overline{w_2(z)},$  where $w_{1}(z)=-\log (1-z)$ and  $w_{2}(z)=-z-\log (1-z)$.
Then the function $u(z)$ defined in Problem \ref{AAP-prob1} takes the form
\be\label{AAP7-eq9}
u(z)=-\overline{z}-2\log |1-z|+(1-|z|^{2})\frac{1}{2}\left( \frac{z}{1-z}+
\frac{\overline{z}^{2}}{1-\overline{z}}\right),
\ee
or equivalently as $u(z)= |z|^2F_1(z)+F_2(z)$, where
$$ F_1(z)=- \frac{1}{2}\left (\frac{z}{1-z}+
\frac{\overline{z}^{2}}{1-\overline{z}}\right )
$$
and
$$ F_2(z) =-\log (1-z) +\frac{z}{2(1-z)} +\overline{\left (-z -\log (1-z) +\frac{z^2}{2(1-z)}\right )}.
$$
Our calculations suggest that $u(z)$ defined by \eqref{AAP7-eq9} is univalent in $\ID$ although we are unable to
prove this at present with a short sketch. However, for a ready reference, the graphs of $u(z)$ for $|z| <1/4$, $|z|<1/2$, $|z| <3/4$ and $|z|<1$ are
shown in Figure \ref{Fig2-u(n)}.
\begin{figure}
\begin{center}
\includegraphics[height=5.0cm, width=5.5cm, scale=1]{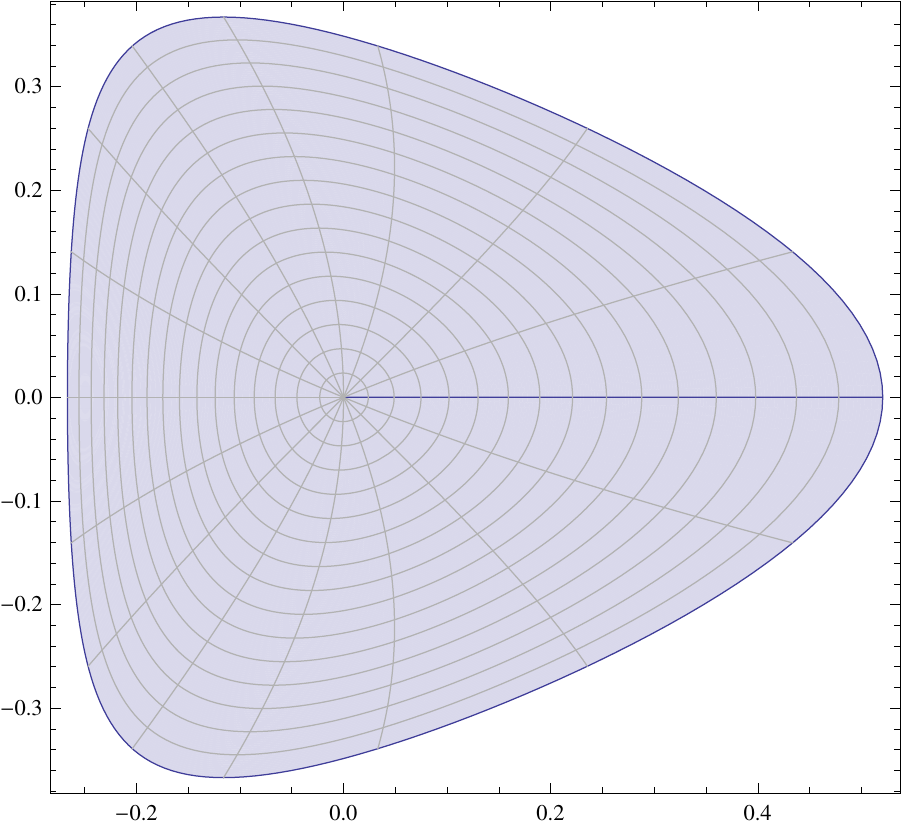}
\hspace{1cm}
\includegraphics[height=5.0cm, width=5.5cm, scale=1]{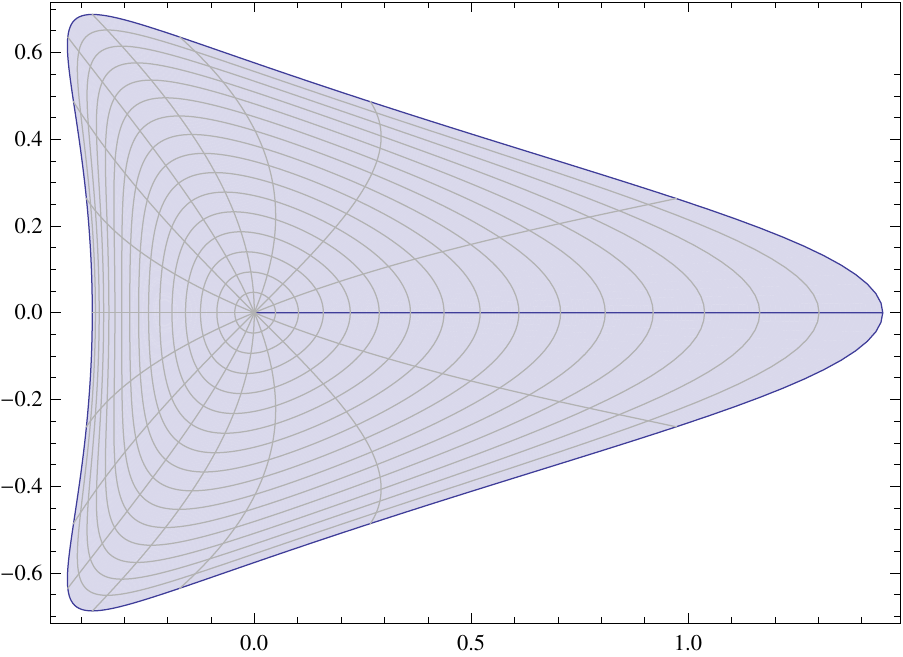}
\end{center}
$u(|z| <1/4)$ \hspace{5cm} $u(|z| < 1/2)$

\vspace{0.25cm}

\begin{center}
\includegraphics[height=5.0cm, width=5.5cm, scale=1]{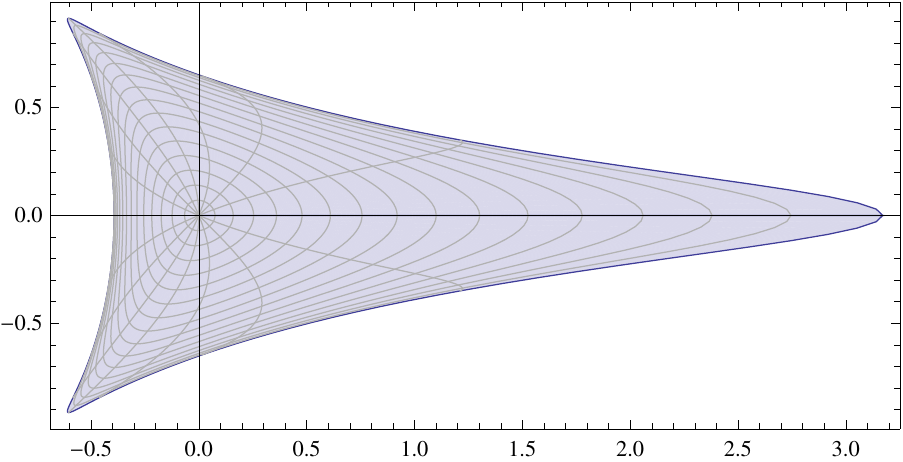}
\hspace{1cm}
\includegraphics[height=5.0cm, width=5.5cm, scale=1]{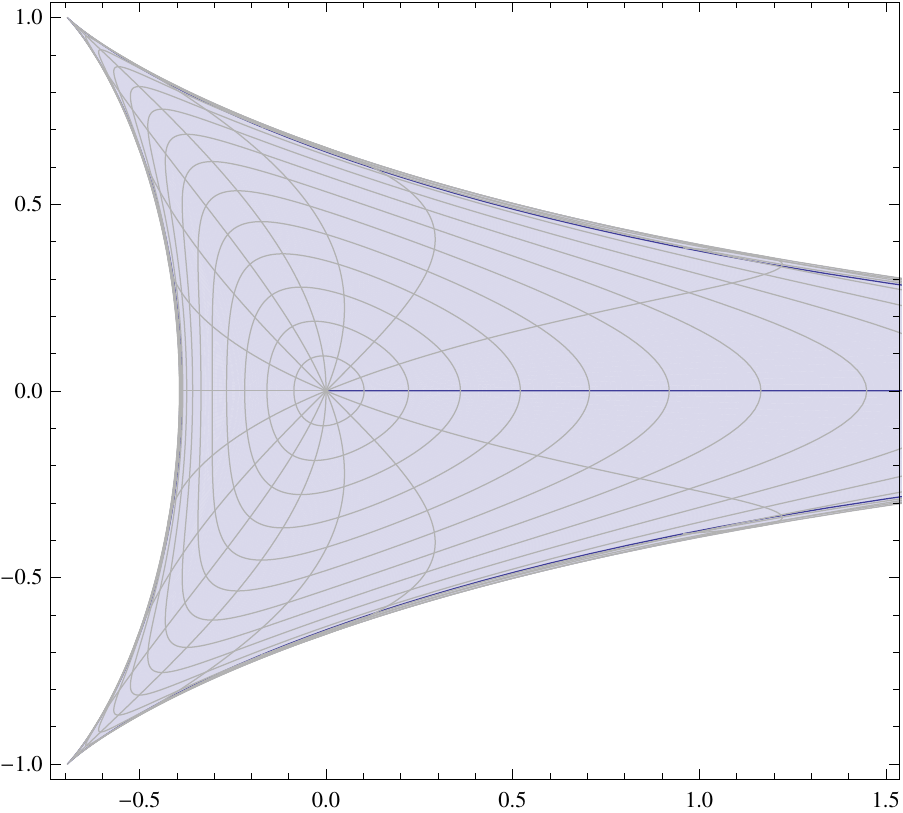}
\end{center}
$u(|z| <3/4)$ \hspace{5cm} $u(|z| < 1)$

\caption{Graphs of $u(z)$ for $|z| <1/4$, $|z|<1/2$, $|z| <3/4$ and $|z|<1$ \label{Fig2-u(n)}}
\end{figure}

\end{example}

We consider the class ${\mathcal C}(\alpha)$ of locally univalent functions $f(z)=z+a_2z^2+\cdots $ in $\ID$
satisfying the condition
$${\rm Re}\left(1+\frac{zf''(z)}{f'(z)}\right)>\alpha,\quad z\in\ID,
$$
where $\alpha \in [-1/2,1)$. For $\alpha=0$, ${\mathcal C}(0)$ represents the usual class of univalent convex mappings in $\ID$.
It is well-known that functions in ${\mathcal C}(-1/2)$, called convex functions of order $-1/2$, are univalent in $\ID$.
Moreover, ${\mathcal C}(-1/2)$ is contained in the class of functions convex in some direction and hence, functions in
${\mathcal C}(-1/2)$ are also close-to-convex in $\ID$. The class  ${\mathcal C}(-1/2)$ has been used to propose
a conjecture by Mocanu \cite{Mocanu-2011} which was later solved
by Bshouty and Lyzzaik \cite{Bshouty-Lyzzaik-2010} (see also  \cite{samy-sairma-pre13} for a general result).
This is another reason we wish to investigate properties of biharmonic mappings connected with the class  ${\mathcal C}(-1/2)$.

\begin{theorem}\label{AAP7-theo1}
Suppose that $u(z)=h(z)+(1-|z|^2)\frac{1}{2}zh'(z)$, where $h(z)=z+\cdots $ is analytic in $\ID$. Then we have the following
\begin{enumerate}
\item[{\rm (1)}] $u$ is sense-preserving in $\ID$ if $h$ is convex (univalent) in $\ID$.
\item[{\rm (2)}] $u$ is sense-preserving in $\ID$ even if $h$ is a convex function of order $\alpha$, $-1/2\leq \alpha <0$, in $\ID$.
In particular, $u$ is sense-preserving for $|z|<1$ if $h$ is a convex function of order $-1/2$ in $\ID$.
\item[{\rm (3)}] $u$ is sense-preserving for $|z|<\sqrt{7}-2 \approx 0.64575$ if $h$ is univalent in $\ID$.
\end{enumerate}
\end{theorem}
\begin{proof}
Set $|z|=r$ and observe that the routine calculations give
\beqq
u_{z}(z)&=&h'(z)+(1-r^{2})\frac{1}{2}\left[ h'(z)+zh''(z)\right] -\frac{1}{2}\overline{z}zh'(z)\\
&=&\frac{1}{2}(3-2r^2)h'(z)+\frac{1}{2}(1-r^{2})zh''(z)\\
&=&\frac{h'(z)}{2}\left[ 2-r^2+ (1-r^{2}) \left (1+\frac{zh''(z)}{h'(z)}\right) \right ]
\eeqq
and $u_{\overline{z}}(z)=-(1/2)z^{2}h'(z)$ so that
\beq\label{AAP7-eq3}
J_u(z)&=&|u_{z}(z)|^{2}-|u_{\overline{z}}(z)|^{2} \nonumber\\
&=&\frac{|h'(z)|^2}{4}\left[ \left |2-r^2+ (1-r^{2}) \left (1+\frac{zh''(z)}{h'(z)}\right)\right|^2-r^4 \right ]
\eeq
which is clearly positive if $h$ is convex in $\ID$. The part (1) follows, but for the proof of the remaining two cases, we need to supply
some details.

Suppose that  $h$ is a convex function of order $\alpha$, $-1/2\leq \alpha <0$.  Then by the definition,
$$1+\frac{zh''(z)}{h'(z)}\prec A(z):=\frac{1+(1-2\alpha) z}{1-z}, ~\mbox{ i.e., }~
\frac{zh''(z)}{h'(z)}\prec \frac{2(1-\alpha)z}{1-z}, \quad z\in\ID,
$$
where $\prec$ is the usual subordination (see for example  \cite[Chapter~6]{Duren:univ}).
It is a simple exercise to see that $A(z)$ maps the disk $|z|< r$ conformally onto the disk
$$\left \{w:\, \left |w- \frac{1+(1-2\alpha)r^2}{1-r^2}\right |<\frac{2(1-\alpha)r}{1-r^2}\right \}
$$
so that ${\rm Re}\, (w)>(1-(1-2\alpha)r)/(1+r).$ This observation shows that
\beqq
2-r^2+ (1-r^2) {\rm Re}\left (1+\frac{zh''(z)}{h'(z)}\right)
&>& 2-r^2+ (1-r^2)\left (\frac{1-(1-2\alpha)r}{1+r}\right)\\
&=&2(1-\alpha)(1-r)+(1+2\alpha) (1-r^2)+r^2
\eeqq
which is positive for all $r\in [0,1)$. Again, it follows that
\beqq
J_u(z)&\geq& \frac{|h'(z)|^2}{4}\left[ 3-2(1-\alpha)r -2\alpha r^2)^2-r^4 \right ]\\
&=&\frac{|h'(z)|^2}{4}[3-2(1-\alpha)r -(2\alpha+1) r^2][3-2(1-\alpha)r -(2\alpha -1)r^2]\\
&=&\frac{|h'(z)|^2}{4}(1-r)(3 +(2\alpha+1)r) [3(1-r)+r(1+2\alpha) +(1-2\alpha)r^2]
\eeqq
showing that $J_u(z)>0$ for $|z|<1$.

In the final case, we suppose that $h$ is univalent in $\ID$. Then from the well-known result (see for instance the proof of
Theorem 3 in \cite[p.~32]{Duren:univ}), it follows that
$$ {\rm Re}\left ( 1+\frac{zh''(z)}{h'(z)} \right)> \frac{1-4r+r^2}{1-r^2} ~\mbox{ for $|z|=r$}
$$
and thus, using the last relation, we find that
$$2-r^2+ (1-r^2) {\rm Re}\left (1+\frac{zh''(z)}{h'(z)}\right)> 3-4r
$$
which is non-negative whenever $r\leq r_1=3/4$. Consequently, for $|z|<r_1$, we find that
\beqq
J_u(z)&\geq& \frac{|h'(z)|^2}{4}\left[ (3-4r)^2-r^4 \right ]\\
&=&\frac{|h'(z)|^2}{4}(3-4r-r^2)(3-4r+r^2)\\
&=&\frac{|h'(z)|^2}{4}(3-4r-r^2)(1-r)(3-r)
\eeqq
showing that $J_u(z)>0$ for $|z|<\sqrt{7}-2 \approx 0.64575$.

The proof of the theorem is complete.
\end{proof}

\begin{theorem}\label{AAP7-theo2}
Suppose that $u(z)=h(z)+(1-|z|^{2})\frac{1}{2}zh'(z)$, where $h(z)=z+\cdots $ is a convex of order $\alpha \in [0,1)$ in $\ID$.
Then the biharmonic mapping $u(z)$ is univalent in the unit disk $\ID$ if $\alpha\geq 1/2$, and in
the subdisk $|z|<r$ if $\alpha \in [0,1/2)$, where
$$r= \frac{-(1-\alpha)+\sqrt{(1-\alpha)^2+1-2\alpha)}}{1-2\alpha}.
$$
\end{theorem}
\begin{proof}
The assumption, in particular, gives that $h$ is convex in $\ID$ and thus, by Theorem \ref{AAP7-theo1}(1), the function $u$
is sense-preserving in $\ID$. We claim that $u$ is univalent in $|z|=r$.

Let $D=h(\ID)$ and $\varphi (z)=zh'(z)$. Define $\psi\colon D\ra \IC$ by $\psi (w)=\big (\varphi \circ h^{-1}\big )(w)$.
Then  $\psi$ is analytic on the convex domain $D$ and
\be\label{AAP7-eq4}
\psi '(w)=\frac{\varphi '(z)}{h'(z)}=1+\frac{zh''(z)}{h'(z)}.
\ee
Now, we suppose that $z_1,z_2\in \ID$, $z_1\neq z_2$, $|z_1|=|z_2|=\rho$ for an arbitrary fixed $\rho$, where $r=|z|<\rho$ and
such that $u (z_1)= u(z_2)$. Also, we set $w_1 = h(z_1)$ and $w_2 = h(z_2)$.
Then $\psi (w_1)=\varphi (z_1)=z_1h'(z_1)$ and $\psi (w_2)=\varphi (z_2)=z_2h'(z_2)$ so that
$$ w_1-w_2= (1-\rho^2)\frac{1}{2}\left( z_2h'(z_2)-z_1h'(z_1)\right)=(1-\rho ^2)\frac{1}{2}(\psi (w_2)-\psi (w_1)) .
$$
On the other hand, because $\psi$ is analytic on the convex domain $D$, this can be
equivalently rewritten as
\be\label{AAP7-eq5}
w_1-w_2 = (1-\rho^2)\int_{[w_1,w_2]} \psi '(w)\,dw ,
\ee
where the integral is taken over a straight-line segment $\Gamma=[w_1,w_2]$ connecting $w_1$ to $w_2$ in the convex domain $D$.
By convexity, $\Gamma \subset h(\overline{\ID}_{\rho})$ and the curve $\gamma =h^{-1}(\Gamma )$, joining the points $z_1$and $z_2$,
lies in the subdisk $\overline{\ID}_{\rho}\subset \ID$.  Consequently,  ${\inf}_{\gamma }(1-r^2)\geq (1-\rho^2)$ and thus, \eqref{AAP7-eq4} and
\eqref{AAP7-eq5} give
\beqq
\left | w_1-w_2\right|&\leq
&(1-\rho^2)\frac{1}{2} \int_{\Gamma }\left |1+\frac{zh''(z)}{h'(z)}\right |\,|dw|\\
&\leq & \frac{1}{2}\int_{\Gamma }(1-|z|^2)\left |1+\frac{zh''(z)}{h'(z)}\right |\,|dw|\\
&\leq & \frac{1}{2}\int_{\Gamma }(1-|z|^2) \left (\frac{1+(1-2\alpha)|z|}{1-|z|}\right ) |dw|\\
&<& \int_{\Gamma } |dw| =\left | w_1-w_2\right|
\eeqq
because, by hypothesis,
$$(1+|z|)(1+(1-2\alpha)|z|)< 2, ~\mbox{ i.e., $(1-2\alpha)|z|^2- 2(1-\alpha)|z|-1 <0$,}
$$
which is not possible. Thus, $u(z_1)\neq u(z_2)$ and this proves the univalency of $u(z)$
on the circle $|z|=\rho$. Since $u$ is sense-preserving in $\ID$, this holds in $|z|<r$ for each $r\leq \rho$,
we find that $u$ is univalent in $|z|<r$. This proves the theorem.
\end{proof}

\begin{corollary}\label{AAP7-cor2}
If $h$ is convex in $\ID$, then  $u(z)=h(z)+(1-|z|^{2})\frac{1}{2}zh'(z)$
is univalent on every disk $|z|<\sqrt{2}-1\approx 0.41421356$.
\end{corollary}
We conjecture that the number $\sqrt{2}-1$ in Corollary \ref{AAP7-cor2} could be improved to $1$.

\begin{corollary}\label{AAP7-cor3}
If $h$ is convex of order $1/2$ in $\ID$, then  $u(z)=h(z)+(1-|z|^{2})\frac{1}{2}zh'(z)$
is univalent in $\ID$.
\end{corollary}

\section{Schwarz lemma for biharmonic mappings}\label{AAP7-sec3}
A well-known harmonic version of the classical Schwarz lemma due to
Heinz \cite{He} (see also \cite{Du-2004}) says the following.

\begin{Thm}\label{AAP7-LemC}
If $f\colon \ID\ra \ID$ is harmonic such that $f(0)=0$, then
$$
|f(z)|\leq\frac{4}{\pi}\arctan|z|\leq\frac{4}{\pi}|z|, \quad z\in \ID.
$$
This inequality is sharp for each $z\in\ID$. Furthermore, the bound is sharp everywhere
$($but is attained only at the origin$)$ for univalent harmonic mappings $f$ of $\ID$ onto itself with $f(0)=0$.
\end{Thm}

Our next aim is to prove a biharmonic version of the classical Schwarz lemma.

\begin{theorem}{\rm (Schwarz lemma for biharmonic mappings)}\label{AAP7-theo4}
Suppose that $H\colon \ID\ra \ID$ is harmonic and $H(z)=w_1(z)+\overline{w_2(z)},$  where both $w_1$ and $w_2$ are
analytic in $\ID$. If $u\colon \ID\ra \ID$ is a biharmonic function of the form
$$u(z)=H(z)+(1-|z|^{2})\frac{1}{2}(zw_1'(z)+\overline{zw_2'(z)})
$$
such that $u(0)=0$, then
$$|u(z)|\leq \frac{4}{\pi }\arctan (|z|)+|z|
~\mbox{ and }~\Lambda _{u}(0)=|u_{z}(0)|+|u_{\overline{z}}(0)|\leq \frac{6}{\pi }.
$$
\end{theorem}
\begin{proof}
The function $H$ has a Poisson representation
$$H(z) =\frac{1}{2\pi}\int_0^{2\pi} {\rm Re}\left ( \frac{1+e^{-i\theta}z}{1-e^{-i\theta}z}\right ) H(e^{i\theta })\,d\theta
=\frac{1}{2\pi}\int_0^{2\pi} \frac{1-|z|^2}{|1-e^{-i\theta}z|^2}H(e^{i\theta })\,d\theta .
$$
Then
$$\frac{r\partial H(z)}{\partial r}=\frac{z\partial H(z)}{\partial z}
=\frac{1}{2\pi }\int_0^{2\pi} \frac{2e^{-i\theta}z}{(1-e^{-i\theta}z)^2} H(e^{i\theta })\,d\theta
$$
and similarly, we have
$$\frac{z\partial H(z)}{\partial \overline{z}}
=\frac{1}{2\pi }\int_0^{2\pi} \frac{2e^{i\theta}z}{(1-e^{i\theta}\overline{z})^2} H(e^{i\theta })\,d\theta.
$$
Thus, we obtain that
$$H_z(0)=\frac{1}{\pi }\int_0^{2\pi} e^{-i\theta} H(e^{i\theta })\,d\theta, \quad
H_{\overline{z}}(0)=\frac{1}{\pi }\int_0^{2\pi} e^{i\theta} H(e^{i\theta })\,d\theta
$$
and
$$(1-|z|^2)\frac{1}{2}\frac{r\partial H(z)}{\partial r}\leq (1-|z|^2)|z|
\frac{1}{2\pi }\int_0^{2\pi} \frac{1}{|1-e^{-i\theta}z|^{2}}|H(e^{i\theta })|\,d\theta .
$$
Finally, we deduce that
\beqq
|u(z)|&\leq& \left |\frac{1}{2\pi }\int_0^{2\pi} \frac{1-|z|^{2}}{|1-e^{-i\theta}
z|^{2}}H(e^{i\theta })\,d\theta \right | +|z|\frac{1}{2\pi }\int _0^{2\pi} \frac{1-|z|^{2}}{|1-e^{-i\theta}z|^{2}}|H(e^{i\theta })|\,d\theta \\
&\leq &\frac{4}{\pi }\arctan (|z|)+|z|,
\eeqq
where the bound in the first term of the last inequality follows from Theorem \Ref{AAP7-LemC}. To prove the second inequality, we note that
\begin{eqnarray*}
u_{z}(z) &=&H_{z}(z)-\overline{z}\frac{1}{2}(zw_1'(z)+\overline{zw_2'(z)}) +\frac{1}{2}\left( 1-|z|^{2}\right) \left (
H_{z}(z)+zH_{zz}(z)\right ) \\
u_{\overline{z}}(z) &=&H_{\overline{z}}(z)-z\frac{1}{2}(zw_1'(z)+\overline{zw_2'(z)})+\frac{1}{2}\left( 1-|z|^{2}\right) %
\left ( H_{\overline{z}}(z)+\overline{z}H_{\overline{zz}}(z)\right )
\end{eqnarray*}
and, because $|H_{z}(0)|+|H_{ \overline{z}}(0)|\leq 4/\pi$, we deduce that
$$\Lambda _{u}(0)=|u_{z}(0)|+|u_{\overline{z}}(0)|=(3/2)(|H_{z}(0)|+|H_{ \overline{z}}(0)|)\leq \frac{6}{\pi }.
$$
This completes the proof of the lemma.
\end{proof}

%
%
%
From the proof of Theorem \ref{AAP7-theo4}, the following result is trivial.

\begin{corollary}
If $u\in {\mathcal F} $ and $u:\,\ID\rightarrow \ID$, then there exists a positive constant $c>0$ such that
$(1-|z|^{2}) (|u_z| +|u_{\overline{z}}|\leq c <\infty.$ That is, $u$ is a Bloch function.
\end{corollary}

\end{document}